\documentclass[12pt]{article} 
\usepackage[sectionbib]{natbib}
\usepackage{array,epsfig,fancyheadings,rotating}
\usepackage[]{hyperref}  
\usepackage{sectsty, secdot}
\sectionfont{\fontsize{12}{14pt plus.8pt minus .6pt}\selectfont}
\renewcommand{\theequation}{\thesection\arabic{equation}}
\subsectionfont{\fontsize{12}{14pt plus.8pt minus .6pt}\selectfont}

\usepackage{amsmath}
\usepackage{amssymb}
\usepackage{amsfonts}
\usepackage{multirow}
\usepackage{amsthm}

\newtheorem{theorem}{Theorem}

\newtheorem{proposition}{Proposition}
\newtheorem{assumption}{Assumption}
\theoremstyle{definition}

\begin{document}


\renewcommand{\baselinestretch}{2}

\markright{ \hbox{\footnotesize\rm 
}\hfill\\[-13pt]
\hbox{\footnotesize\rm
}\hfill }

\markboth{\hfill{\footnotesize\rm Sneha Jadhav AND Shuangge Ma} \hfill}
{\hfill {\footnotesize\rm  Functional Measurement Error in Functional Regression} \hfill}

\renewcommand{\thefootnote}{}
$\ $\par


\fontsize{12}{14pt plus.8pt minus .6pt}\selectfont \vspace{0.8pc}
\centerline{\large\bf Functional Measurement Error in Functional Regression}
\vspace{.4cm} \centerline{Sneha Jadhav, Shuangge Ma} \vspace{.4cm} \centerline{\it
Yale University} \vspace{.55cm} \fontsize{9}{11.5pt plus.8pt minus
.6pt}\selectfont


\begin{quotation}
\noindent {\it Abstract:}  Measurement error is an important problem that has not been very well studied in the context of Functional Data Analysis. To the best of our knowledge,  there are no existing methods that address the presence of functional measurement errors in generalized functional linear models.  A framework is proposed for estimating the slope function in the presence of measurement error in the generalized functional linear model with a scalar response. This work extends the conditional-score method to the case when both the measurement error and the independent variables lie in an infinite dimensional space. Asymptotic results are obtained for the proposed estimate and its behavior is studied via simulations, when the response is continuous or binary. It's performance on real data is demonstrated through a simulation study based on the Canadian Weather data-set, where errors are introduced in the data-set and it is observed that the proposed estimate indeed performs better than a naive estimate that ignores the measurement error. 

\vspace{9pt}
\par
\end{quotation}\par

\def\thefigure{\arabic{figure}}
\def\thetable{\arabic{table}}

\renewcommand{\theequation}{\thesection.\arabic{equation}}

\fontsize{12}{14pt plus.8pt minus .6pt}\selectfont

\section{Introduction}
Measurement error in multivariate data is a well-studied problem, and consequently there are multiple ways to address it Carroll et al.\ (2004). This problem arises in many diverse fields like nutrition, environmental studies and so on Carroll \& Raymond \ (1998). It is well established that ignoring this error can lead to several problems like bias in the estimation of regression parameters. For a detailed discussion on the repercussions of ignoring measurement error, refer to Carroll et al.\ (2004). Measurement error also arises in functional data where a large number of repeated measurements for variables are available. It is only natural to assume that if there is error in the data, it is present at all of the repeated measurements i.e. the measurement error is a functional variable. In the Functional Data Analysis literature, there are limited tools available to handle measurement errors. Most of the available literature assumes measurement error at the discrete points at which the functions are observed (Yao, M{\"u}ller \& Wang, 1998; Cardot, Crambes et al. 2007 ; James, 2002; Crambes, Kneip \& Sarda, 2009; Goldsmith et al., 2011 ,Goldsmith, Wand \& Crainiceanu, 2011). Specifically, denote the observed values of the function $X(\cdot)$ at a grid of points $a\leq t_1<...<t_m\leq b$ by $X_1,...,X_m$. All of these works assume $X_k=X(t_k)+e_k,\,k=1,...,m$ and the errors $e_k$ are indepedent or uncorrelated. This assumption on errors is very restrictive, and the asymptotics as well as the performance of these methods depends on the validity of these assumptions. A recent work Cai (2015), allows for correlation between the measurement errors however, imposes parametric structures on the covariance. To the best of our knowledge, there is only one work that considers the measurement error to be functional in nature (Chakraborty \& Panaretos, 2017). The measurement error model considered is $ W(\cdot)=X(\cdot)+U(\cdot),$ where $U(\cdot)$ is a measurement error stochastic process. In order to avoid identifiability issues, certain conditions are imposed. It is assumed that the measurement error process is at a much finer scale than the true covariate. This is achieved by imposing the following two conditions: 1) there exists $\delta>0$, such that $cov(U(t),U(s))=0$ if $|s-t|>\delta $ and 2) the covariance operator of $X(\cdot)$ is analytic on an open set containing $|s-t|\leq \delta$. Estimation of the covariance function of the error process under these conditions, is given in Descary \& Panaretos (2016). The assumptions are more general than those made previously but still quite restrictive. Moreover, none of the methods can accommodate measurement error process with correlations in a regression model with a binary response.

To fill the knowledge gap, we develop a framework that will allow for measurement error with a more general correlation structure and response that is binary or has a normal distribution. This proposed framework is based on the conditional-score method proposed in Stefanski \& Carroll (1987). We use the Karhunen-Lo{\'e}ve expansion to obtain estimating equations similar to those in Stefanski \& Carroll (1987). However, in our case the number of parameters diverge, making this framework and involved asymptotics non-trivial. The simulations presented later, demonstrates how measurement errors lead to an in-correct estimate of the slope function.

\setcounter{equation}{0} 
\section{Methods}
Given covariate $X(t),\,t\in L^2=L^2[a,b],$ assume 
that $Y$ has the distribution  $f_{Y}\{y;\theta_1,x(\cdot)\}$ with respect to a dominating measure $h$ given by:
\begin{align}  \text{exp}\left[\dfrac{y\{\tilde{\beta}_0+\int_a^b\tilde{\beta}(t)x(t)d(t)\}-b\{\tilde{\beta}_0+\int_a^b\tilde{\beta}(t)x(t)dt\}}{a(\tilde{\phi})}+c(y,\tilde{\phi})\right]  \label{em1}.
\end{align}
All integrals hereafter are taken over $[a,b]$ and , $\theta_1=(\tilde{\beta}_0,\tilde{\beta}(\cdot),\tilde{\phi})$. Refer to \cite{stefanski1987conditional} for details on all the distributions that are included in the above model.

Let $(\rho_k)_{k=1}^{\infty}$ be orthonormal basis functions in the $L^2$ space. Using basis expansion, we obtain
\begin{align*}
\int\tilde{\beta}(t)X(t)dt=\sum\limits_{k=1}^{\infty} X_{k}\tilde{\beta}_k,
\end{align*}

where $X_{k}=\int X(t)\rho_k(t) dt$, $\tilde{\beta}_k=\int \tilde{\beta}(t)\rho_k(t) dt, k\geq1$. Thus, model \eqref{em1} has infinitely many parameters. We address this issue of infinite dimension with a truncation strategy. Denote $\tilde{\beta}=(\tilde{\beta}_1,...,\tilde{\beta}_{p_n})^{'} $ and $\tilde{\theta}=(\tilde{\beta}_0,\tilde{\beta},\tilde{\phi})$. Instead of the model \eqref{em1}, we work with the following sequence of models with increasing dimension $p_n$:
\begin{align}  f_{Y}(y;\tilde{\theta},x)&=\text{exp}\left\{\dfrac{y\left(\tilde{\beta}_0+\sum\limits_{k=1}^{p_n} x_{k}\tilde{\beta}_k\right)-b\left(\tilde{\beta}_0+\sum\limits_{k=1}^{p_n} x_{k}\tilde{\beta}_k\right)}{a(\tilde{\phi})}+c(y,\tilde{\phi})\right\} , \label{m2}
\end{align}
where $p_n \to \infty$ as $n \to \infty$.
Due to measurement error, we observe surrogate variable $W(\cdot)$ as opposed to the true covariate $X(\cdot)$ . Assume 
\begin{align}
W(\cdot)=X(\cdot)+U(\cdot), \label{me}
\end{align}
where $U(\cdot)$ is a Gaussian process with mean function $0$ and covariance function $K(\cdot,\cdot)$. Take $(\rho_i)_{i=1}^{\infty}$ to be the basis constructed from the eigenfunctions of the integral operator $\mathcal{K}$ associated with the covariance function $K(\cdot,\cdot)$.
Denote $W_{k}=\int W(t)\rho_k(t) dt$ and $U_{k}=\int U(t)\rho_k(t) dt$. This yields the following measurement error set-up for \eqref{m2}: $$W_{k}=X_{k}+U_{k},\, 1\leq k\leq p_n,$$ 
where $U_{k}$'s are independent and $U_{k}\sim N(0,\lambda_k)$ with $\lambda_k$ being the eigenvalue associated with the $k^{th}$ eigenfunction of $\mathcal{K}$. 
This set-up is similar to that in \cite{stefanski1987conditional} where the authors proposed the conditional-score based sufficiency estimator for generalized linear models. For a sample of size $n$, assume that $X_i(\cdot),\,i=1,...,n$ are deterministic and denote $\mathbf{X}_i=(X_{i1},...,X_{ip_n})^{'}$ and $\mathbf{W}_i=(W_{i1},...,W_{ip_n})^{'}$. Let $\widetilde{W}_i^{(l)}(\cdot)$ denote the $l^{th}$ replicate of    
the function $W_i(\cdot) $ and $\widetilde{W}_{ik}^{(l)}=\int \rho_k(t)\widetilde{W}_i^{(l)}(t)dt$. Let $\widetilde{\mathbf{W}}_{i}^{(l)}=(\widetilde{W}_{i1}^{(l)},...,\widetilde{W}_{ip_n}^{(l)})'. $ The likelihood of  $\widetilde{W}_i=(\widetilde{\mathbf{W}}_{i}^{(1)},...,\widetilde{\mathbf{W}}_{i}^{(m)})'$ is
\begin{align}
f_{\widetilde{W}_i}(\widetilde{w}_i,\tilde{\theta},\mathbf{x}_i)&= \prod\limits_{j=1}^{m} \dfrac{(2\pi)^{-p_n/2}}{|\Omega_1|}\text{exp}\left\{ \dfrac{-1}{2}(\widetilde{\mathbf{w}}_i^{(j)}-\mathbf{x}_i)^T(\widetilde{\mathbf{w}}_i^{(j)}-\mathbf{x}_i)\right\}, \label{xdensity}
\end{align}
where, $\Omega_1=\text{diag}(\lambda_1,...,\lambda_{p_n})$ is the unknown covariance matrix of the measurement error vector and $m$ denotes the number of replicates. Without loss of generality we assume that $m=1$ and $\Omega=\Omega_1/a(\tilde{\phi})$ is known. In practice we use the estimate of $\Omega_1$ and $a(\tilde{\phi})$ that are given subsequently. The estimate of $\Omega_1$ is obtained from the estimate of the covariance function $K(\cdot,\cdot)$. We can use the method in Chakraborty \& Panaretos (2017) to estimate this covariance function. In this case, the covariance function has to satisfy the two assumptions mentioned in the previous section.  If not, then additional data like replicates is required so that the covariance function can be estimated without these assumptions. Thus, there is always trade-off between additional data and assumptions in the measurement error framework. Most methods related to functional data assume that there is no measurement error, a few that relax this assumption make very strong assumptions regarding its structure. We relax these stringent assumptions in presence of replicates. Thus, if there is any reason to suspect measurement error, a sound strategy is to collect additional required information right at the outset instead of relying on unverifiable assumptions regarding measurement error. Let $Y, \, \mathbf{W}$ denote random variables with the same distribution as $Y_i,\,\mathbf{W}_i,\, i=1,...,n $ respectively.  We refer to the variable $Y$ as the response variable. For now, we drop the subscript $i$ and use it when necessary. We assume that $\mathbf{W}$ has no information on $Y$ other than what is contained in $\mathbf{X}$ that is,
\begin{align*}
f_{Y,\mathbf{W}}(y,\mathbf{w};\tilde{\theta},\mathbf{x})=f_Y(y;\tilde{\theta},\mathbf{x})f_\mathbf{W}(\mathbf{w},\tilde{\theta},\mathbf{x}).
\end{align*}

Let $\mathbf{\Delta}(\beta)=\mathbf{W}+Y\Omega\beta=(\Delta_{1}(\beta),...,\Delta_{p_n}(\beta))^{'},\, \mathbf{\Delta}(\tilde{\beta})=\mathbf{\Delta}=(\Delta_{1},...,\Delta_{p_n})^{'}$. The distribution $f_{Y|\mathbf{\Delta}}(y|\mathbf{\Delta}=\mathbf{\delta};\tilde{\theta})$ of  $Y|\mathbf{\Delta}$ is given by
\begin{align}
\text{exp}\left[y\tilde{\eta}-\dfrac{1}{2}y^2 \tilde{\beta}^{'}\Omega\tilde{\beta}/a(\tilde{\phi})+c(y,\tilde{\phi})-\text{log}\left\{S(\tilde{\eta},\tilde{\beta},\tilde{\phi})\right\} \right],\label{cond}
\end{align}	
where $S(\tilde{\eta},\tilde{\beta},\tilde{\phi})=\displaystyle \int \text{exp}\left\{y\tilde{\eta}-\dfrac{1}{2}y^2 \tilde{\beta}^{'}\Omega\tilde{\beta}/a(\tilde{\phi})+c(y,\tilde{\phi})\right\}dh(y)$, $\tilde{\eta}=(\tilde{\beta}_0+\mathbf{\delta}^{'}\tilde{\beta})/a(\tilde{\phi})$.
This distribution belongs to the exponential family. Let $\theta=(\beta_0,\beta,\phi)^{'},\, \beta=(\beta_1,...,\beta_{p_n})^{'}, \\ \Psi(y,\mathbf{w},\theta)=(\partial/\partial \theta)log \, f_{Y/\mathbf{\Delta}}(y|\mathbf{\delta};\theta),\,\mathbf{\delta}(\beta)=\mathbf{w}+y\Omega\beta $. We use the following estimating equation to estimate $\theta$: 
\begin{align}\sum_{i=1}^{n}\Psi(Y_i,\mathbf{W}_i,\theta)=\mathbf{0}, \label{est1}
\end{align}
where $\Psi(Y_i,\mathbf{W}_i,\theta)$ is,
\begin{align*}
\begin{bmatrix}
\left[Y_i-E(Y_i|\Delta(\beta)=\mathbf{\delta}_i(\beta))\right]/a\left(\phi\right) \\
\left[Y_i-E(Y_i|\mathbf{\Delta}_i(\beta)=\mathbf{\delta}_i(\beta))\right]\dfrac{\mathbf{\delta}_i(\beta)}{a(\phi)}-\left[y_i^2-E(Y_i^2|\mathbf{\Delta}_i(\beta)=\mathbf{\delta}_i(\beta))\right]\dfrac{\Omega\beta}{a(\phi)} \\
r(Y_i,\mathbf{W}_i,\theta)-E(r(Y_i,\mathbf{W}_i,\theta)|\mathbf{\Delta}_i(\beta)=\mathbf{\delta}_i(\beta))
\end{bmatrix} ,
\end{align*}
$r(Y_i,\mathbf{W}_i,\theta)=\dfrac{\partial c(Y_i,\theta)}{\partial\phi}-Y_i\dfrac{\beta_0+\mathbf{\delta}_i(\beta)^{'}\beta}{a^2(\phi)}a'(\phi)+Y_i^2\dfrac{\beta^{'}\Omega\beta}{2a^2(\phi)}a'(\phi),\, a^{'}(x)$  denotes the derivative of function $a$ and $\mathbf{0}$  here and hence onwards denotes a  column vector containing $0$'s of appropriate dimension.
Note that the estimating equation \eqref{est1} is unbiased for $\theta$ and the corresponding estimator does not maximize the conditional likelihood. We study this estimator in greater details when $Y$ is binary and gaussian.


\subsection{Binary Response}
When $Y$ is binary, the truncated model is $P_{\tilde{\theta}}(Y=1|\textbf{X})=
F(\tilde{\beta}_0+\textbf{X}^{'}\tilde{\beta}),\,F(t)=1/(1+e^{-t})$. For this case, it is easy to see that \eqref{cond} gives $P_{\tilde{\theta}}(Y=1|\Delta=\delta)=F\{\tilde{\beta}_0+(\delta-0.5\Omega\tilde{\beta})^{'}\tilde{\beta}\} $. Combining this with \eqref{est1}, we get  the following estimating equations

\begin{align*}
\sum_{i=1}^{n}\Psi(Y_i,\mathbf{W}_i,\theta)=\sum_{i=1}^{n}[Y_i-F[\beta_0+\{\delta_i(\beta)-0.5\Omega\beta\}^{'}\beta]\begin{bmatrix}
1\\
\delta_i(\beta)-\Omega\beta\end{bmatrix}=\begin{bmatrix}
0\\
\textbf{0}\end{bmatrix}.
\end{align*}
Let
\begin{align*}
&\beta_c=(\beta_0,\beta)^{'},\, X_c=(1,\textbf{X}^{'})^{'},\,W_c=(1,\textbf{W}^{'})^{'},\, U_c=(0,U^{'}),\, \\&\Omega_c=\begin{bmatrix} 0, \textbf{0}'\\
\textbf{0}, \Omega \end{bmatrix},\, \delta_c(\beta_c)=W_c+Y\Omega_c\beta_c .
\end{align*}
 Using these notations we can rewrite $U(\beta_c)=\sum_{i=1}^{n}\Psi(Y_i,\mathbf{W}_i,\theta)=\textbf{0}, $  as
\begin{align}
U(\beta_c)=\sum_{i=1}^{n}[Y_i-F\{\delta_{ci}(\beta_c)^{'}\beta_c-0.5\beta_c^{'}\Omega_c\beta_c\}]\begin{bmatrix}
\delta_{ci}(\beta_c)-\Omega_c\beta_c\end{bmatrix}=\textbf{0}. \label{b_score}
\end{align}

We now state the assumptions needed to show the existence and consistency of the solution to this set of equations. We denote the Frobenius norm for a matrix and euclidean norm for vectors by $ \|\cdot \|$ , and a positive constant by $c$.

\begin{assumption}\label{beta}
	Assume that	$ \int \tilde{\beta}^2(t) dt < \infty $. Then $\|\tilde{\beta}\| < m_1< \infty.$
\end{assumption}

\begin{assumption}\label{x}
	Assume that $ \int X_i^2(t)dt < m_2 <\infty.$
\end{assumption}

\begin{assumption}\label{omega}
	Let	$ \| \Omega \| <  \infty $.
\end{assumption}

\begin{assumption}\label{rateb}
	Let	$ p_n^4/n \rightarrow 0 $.
\end{assumption}

\begin{assumption}\label{eigen}
	Let $B=\lambda_{min}\left(\sum_{i=1}^{n} W_{ic}W'_{ic}/n \right)$. Then
	$$\lambda_{max}(\Omega)\text{exp}(\lambda_{max}(\Omega)m_1) \leq B \dfrac{ \text{exp}\left(\underset{1\leq i \leq n}{inf} W_{ic}^{'}\tilde{\beta}-\underset{1\leq i \leq n}{sup}  W_{ic}^{'}\tilde{\beta}\right)}{1+\text{exp}\left(\underset{1\leq i \leq n}{inf} W_{ic}^{'}\tilde{\beta}\right)} \quad{a.s.}   $$
\end{assumption}

In the rest of the paper supremum and infimum are taken over $i=1,...,n$ when not indicated explicitly.
Without loss of generality we assume that the function $\bar{X} =\sum_{i=1}^{n} n^{-1} X_i(t)=0$. This can be achieved by replacing $X_i(\cdot)$ by $X_i^c(\cdot)=X_i(\cdot)-\bar{X}(\cdot)$ and adjusting the intercept in the \eqref{em1}. The corresponding  adjusted measurement error model is $W_i^c(\cdot)=X_i^c(\cdot)+U_i(\cdot)$, where $W_i^c(\cdot)=W_i(\cdot)-\sum_{i=1}^{n} n^{-1} E\{W_i(t)\} $. Let $W_{ik}^c=\int W_i^c(t)\rho_{k}(t) dt,\,X_{ik}^c=\int X_i^c(t)\rho_{k}(t) dt $ and $\overline{W}_k=\sum_{i=1}^{n} n^{-1}E(W_{ik}).$ The adjusted error model leads to $W_{ik}^c=X_{ik}^c+U_{ik}$, $W_{ik}^c=W_{ik}-\overline{W}_k.$ Thus, we need to replace the observed variables $W_{ik}$ by $W_{ik}-\overline{W}_k$ for $k=1,...,p_n$. Using law of large numbers for independent variables we can obtain a consistent estimate of $\overline{W}_k$ as $\sum_{i=1}^{n} n^{-1}W_{ik}$ . 

The following theorem is sufficient to prove weak consistency and existence of the estimator. This follows from Theorem 6.3.4 of \cite{ortega1970iterative}.
\begin{theorem}\label{con} For all $\epsilon>0$, there exists a constant $\zeta>0$ such that for sufficiently large $n$,
	$$ P\left(\underset{|| \beta_c-\tilde{\beta}_c||= \zeta \sqrt{p_n/n}}{\sup}(\beta_c-\tilde{\beta_c})'U(\beta_c)<0\right)\geq 1-\epsilon.$$
\end{theorem}

\textbf{Proof}
Using Taylors Theorem,
\begin{align*}
(\beta_c-\tilde{\beta}_c)^{'}U(\beta_c) &=(\beta_c-\tilde{\beta}_c)^{'} U(\tilde{\beta}_c)+ (\beta_c-\tilde{\beta}_c)^{'}J(\beta_c^{*})(\beta_c-\tilde{\beta}_c) \nonumber\\
&=A_1+A_2, 
\end{align*}
where $\beta_c^{*}$ lies between $\beta_c$ and $\tilde{\beta}_c$ i.e. $max(\|\beta_c^{*}-\beta_c\|,\|\beta_c^{*}-\tilde{\beta}_c)\|\leq \|\beta_c-\tilde{\beta}_c\|$.
From Proposition \ref{score}  we get
\begin{align*}
\underset{|| \beta_c-\tilde{\beta}_c||= \zeta \sqrt{p_n/n}}{\sup}	A_1 = \zeta O_p(\sqrt{n})\sqrt{p_n/n}=\zeta O_p(\sqrt{p_n}). 
\end{align*}
From Proposition \ref{prop3} and Proposition \ref{prop4},
\begin{align*}
\underset{|| \beta_c-\tilde{\beta}_c||= \zeta \sqrt{p_n/n}}{\sup}	A_2&=(\beta_c-\tilde{\beta}_c)^{'}\{J(\beta_c^{*})-J(\tilde{\beta_c})\}(\beta_c-\tilde{\beta}_c)\\
&\quad +(\beta_c-\tilde{\beta}_c)^{'}J(\tilde{\beta}_c)(\beta_c-\tilde{\beta}_c) \nonumber \\
& \leq \zeta^2 p_n^2/\sqrt{n}  -c \zeta^2p_n+\zeta^2o_p(p_n) \nonumber \\
&\leq \zeta^2 o(1) -c \zeta^2p_n+\zeta^2o_p(p_n).
\end{align*}	
Thus, $\underset{|| \beta_c-\tilde{\beta}_c||= \zeta \sqrt{p_n/n}}{\sup}(\beta_c-\tilde{\beta_c})'U(\beta_c) $ is asymptotically dominated by $ -c \zeta^2p_n$ and the result is proved.

\subsection{Gaussian Response}
We now consider the case where $Y$ has normal distribution. In this case, $ \Psi(Y_i,\mathbf{W}_i,\theta) $ in equations \eqref{est1} can be written as
\begin{align}
\Psi(Y_i,\mathbf{W}_i,\theta)= \begin{bmatrix}
\dfrac{1}{\sigma^2}(Y_i-\mu_i)  \label{c_est3}\\
\dfrac{\Omega\beta}{1+\beta'\Omega\beta}-\dfrac{1}{\sigma^2}\left[(Y_i-\mu_i)^2\Omega\beta-(Y_i-\mu_i)\{\mathbf{\delta}_i(\beta)-2\mu_i\Omega\beta\}\right]\\
\dfrac{-1}{2\sigma^2}+\dfrac{(Y_i-\mu_i)^2(1+\beta'\Omega\beta)}{2\sigma^4}             \end{bmatrix}.
\end{align}
Denote identity matrix by $I$. Let $\mathbf{\Delta}_i^*(\beta)=(I+\Omega\beta\beta')^{-1}\{\mathbf{\Delta}_i(\beta)-\beta_0\Omega\beta\}$ and $\mu_i=(\beta_0+\beta'\mathbf{\delta}_i(\beta))/(1+\beta'\Omega\beta)$. Consider the following equations:

\begin{align}
&\sum\limits_{i=1}^{n} (\mathbf{\Delta}_i^*(\beta)Y_i-\mathbf{\Delta}_i^*(\beta)\beta_0-\mathbf{\Delta}_i^*(\beta)\beta'\mathbf{\Delta}_i^*(\beta))=0 , \label{c_est4}\\
&\sum\limits_{i=1}^{n} (Y_i-\beta_0-\beta'\mathbf{\Delta}_i^*(\beta))=\textbf{0}, \nonumber\\
&\sigma^2=\dfrac{1+\beta'\Omega\beta}{n}\sum\limits_{i=1}^{n}(Y_i-\mu_i)^2. \nonumber
\end{align}
The solution to equations \eqref{c_est4} is also a solution to equations \eqref{c_est3}. Note that the above equations are non-linear. Let $\mathbf{W}^c=\mathbf{W}-\mathbf{\overline{W}},Y^c=Y-\overline{Y}$, where $\overline{Y}$ and $\overline{W}$ denote the average of $Y_1,...,Y_n$ and $\mathbf{W}_1,...,\mathbf{W}_n$ respectively. Then, the solution to equations \eqref{c_est4} is also a solution to the following equations:

\begin{align}
&U(\beta)=-\sum\limits_{i,j=1}^{n} \mathbf{W}_i^cY_i^c\beta^T\Omega\beta + \sum\limits_{i,j=1}^{n}Y^{c2}_i\Omega\beta-\sum\limits_{i,j=1}^{n}\mathbf{W}_i^c\mathbf{W}_i^{cT}\beta+\sum\limits_{i,j=1}^{n}\mathbf{W}_i^cY_i^c=\textbf{0} \nonumber \\
&\beta_0=\overline{Y}-\beta\overline{\mathbf{W}} ,\sigma^2=\dfrac{1+\beta^T\Omega\beta}{n}\sum\limits_{i=1}^{n}(Y_i-\mu_i)^2. \label{c_est5}
\end{align}

In addition to  Assumptions \ref{beta},\,\ref{x},\,\ref{omega}, we need the following assumptions to show that the equations \eqref{c_est5} has a solution $\hat{\beta}$ and that this solution is consistent.

\begin{assumption}\label{vary} Assume  $\lambda_{max}(\Omega_1) \leq \lambda_{min}\left(\sum\limits_{i=1}^{n}\mathbf{X}_i^c\mathbf{X}_i^{c'}/n\right)+\lambda_{min}(\Omega_1)$.
\end{assumption}

\begin{assumption}\label{rate} 
Let	$ p_n^3/n \to 0.$
\end{assumption}

\begin{theorem}\label{c_con} For all $\epsilon>0$, there exists a constant $\zeta>0$ such that for a sufficiently large $n$,
	$$ P\left(\underset{|| \beta-\tilde{\beta}||= \zeta \sqrt{p_n/n}}{\sup}(\beta-\tilde{\beta})^TU(\beta)<0\right)\geq 1-\epsilon.$$
\end{theorem}
\textbf{Proof}
Let $\beta^*$ be such that $\text{max}(\|\beta^*-\tilde{\beta}\|,\|\beta^*-\beta\|) \leq \|\beta-\tilde{\beta}\|$. 
Using Taylor's expansion we obtain,
\begin{align*}
(\beta-\tilde{\beta})^T U(\beta )&=(\beta-\tilde{\beta})^T U(\tilde{\beta})+(\beta-\tilde{\beta})^TJ(\beta^*)(\beta-\tilde{\beta})\\
&= (\beta-\tilde{\beta})^T U(\tilde{\beta})+(\beta-\tilde{\beta})^T\{J(\beta^*)-J(\tilde{\beta})\}(\beta-\tilde{\beta})+(\beta-\tilde{\beta})^TJ(\tilde{\beta})(\beta-\tilde{\beta}).
\end{align*}	
From Proposition \ref{c_prop1},\,\ref{c_prop2},\,\ref{c_prop3},
\begin{align*}
\underset{|| \beta-\tilde{\beta}||= \zeta \sqrt{p_n/n}}{\sup}	(\beta-\tilde{\beta})^T U(\beta ) \leq O_p(\sqrt{p_n})\zeta+\zeta^2\dfrac{p_n^{1.5}}{\sqrt{n}}-c\zeta^2p_n+ \zeta^2o_p(p_n).
\end{align*}
Using Assumption \ref{rate} and with an appropriate choice $\zeta$ the theorem holds.



\setcounter{equation}{0} 

\section{Simulations}
In this section, we compare performance of the proposed estimator with alternatives in the presence of measurement errors with various covariance structures and sample sizes. We first consider a case with Gaussian scalar response and then a binary one. In the following, $P(k)$ denotes a Poisson distribution with parameter $k$, $N(a,b)$ denotes normal distribution with mean $a$ and variance $b$.

True covariate function $X_i(t),t\in [0,1], i=1,...,n$ is generated using the Fourier basis as $$X_i(t)=\sum_{k=1}^{p}\varepsilon_{ik} \rho_{k}(t),\, i=1,...,n,\, p=2[n^{(1/5)}], \varepsilon_{ik} \sim P(2) . $$
We study the effect of measurement error with two covariance structures. In Setting $1$, we generate a centered Gaussian Process with covariance function $K(s,t)=\sigma_1 \text{exp}\{-(s-t)^2/(2l^2)\}$. This is a squared exponential function where the covariance depends on the distance between the points. We vary the value of $l$ which controls the range of dependence. In Setting $2$, we use covariance function of the Brownian Bridge: $K(s,t)= \sigma_2\{\text{min}(s,t)-st\}   $. The parameters $\sigma_1,\,\sigma_2$ are introduced to control the level of noise. For $\sigma_2=1$, we obtain the Brownian Bridge. The observed covariate is $W_i(t)=X_i(t)+U_i(t),\, i=1,...,n$. The effect function $\tilde{\beta}(\cdot)$ is generated using the Fourier basis as $\tilde{\beta}(t)=\sum_{k=1}^{p} \beta_k\rho_{k}(t),$ where $\beta_k=k^{-1}.$ 
The Gaussian response is generated as $ Y_i \sim N(\int X_i(t)\tilde{\beta}(t)dt,1)$. For this case, $\sigma_1=5$ in Setting 1. The binary response is generated from a Binomial distribution with $P(Y_i=1)=\text{exp}\{ \text{max}_{t_k}X_i(t_k)\}/[1+ \text{exp}\{\text{max}_{t_k}X_i(t_k)\}]$. For this case, we consider $\sigma_1=2$ in Setting 1. 

To implement our proposed method, we first need to estimate the error covariance structure. For this, we use 50 replicates of the function $W_i(\cdot), i=1,...,n$ denoted by $\widetilde{W}_{ij}(\cdot),\,i=1,...,n,\,j=1,...,50 $. Let $ \widetilde{W}_{i.}(\cdot)$ denote the mean function of $\widetilde{W}_i(\cdot)$. The estimate of the covariance function is considered to be
$$ \hat{K}(s,t)= \dfrac{\sum\limits_{i=1}^{n}\sum\limits_{j=1}^{50}(\widetilde{W}_{ij}(t)-\widetilde{W}_{i.}(t))(\widetilde{W}_{ij}(s)-\widetilde{W}_{i.}(s))}{n(50-1)}.$$

We next consider the problem of selecting the number of components  $p_n$ in the model. Cross-validation, which is a popular method to determine $p_n$ is biased in the presence of measurement error \citep{datta2017cocolasso}. Including a large number of components reduces the loss of information, on the other hand in the presence of measurement error, adds to the total measurement error in the model. Moreover, we use the Newton Raphson algorithm to solve the function $U(\beta_c)=0$ which involves inverting the derivative of $U(\beta)$. From the simulations, we observed that selecting $p_n\leq p$ helps to avoid singularity issues. Thus, ideally we should use $p_n=p$, however $p$ is unknown in practice. We choose $p_n$ as the threshold beyond which the proportion of variation explained by the first $p_n$ components levels off as indicated in Figure 1. We observe that, method can lead to an accurate determination of $p$ when the measurement error is not too large.   

\begin{figure}[h!]
\centerline{\epsfig{file=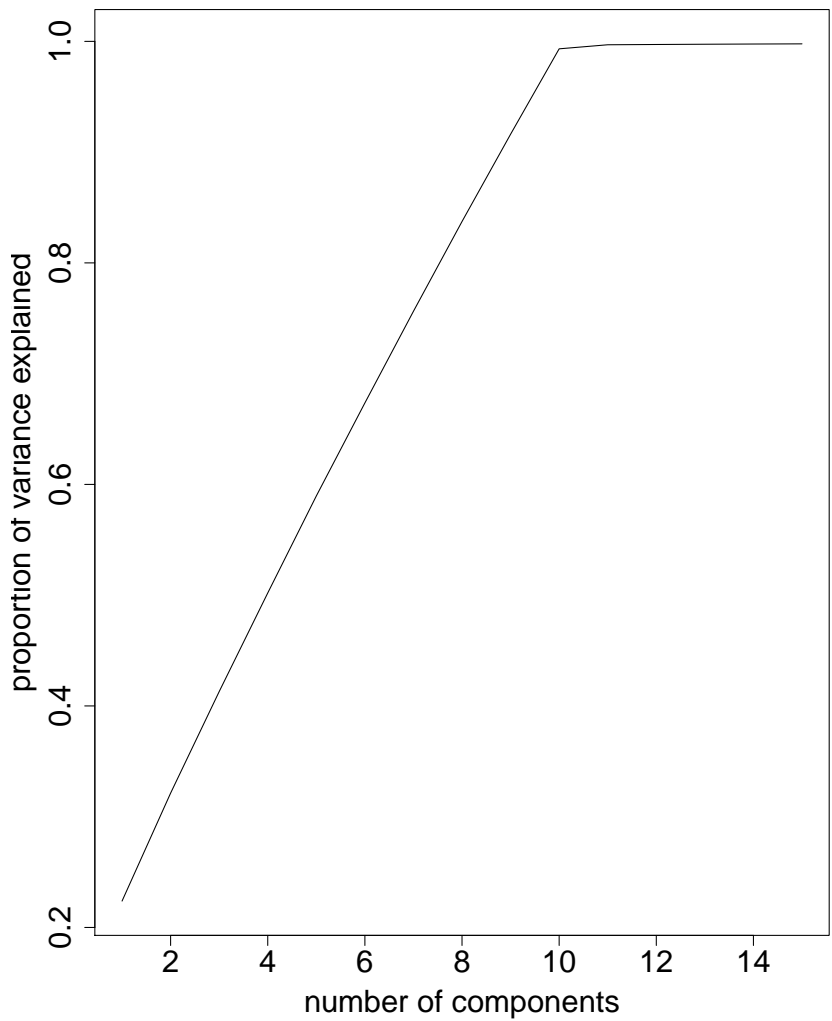,width=4.5in}}\par  
\caption{Selection of $p_n$ via the proportion of variance explained.}
\end{figure} 


We use the Newton Raphson algorithm to obtain the solution to equations \eqref{b_score} and \eqref{c_est5}. Note that these equations have multiple solutions. It is not clear which solution to choose unlike in maximization or minimization problems. To address this, we use the naive estimator as the initial value for the Newton Raphson algorithm. This leads to an accurate estimate if the naive estimate is close to the true value, in other words if the measurement error is small. This naive estimate, $\hat{\beta}^{naive}(\cdot) $ is obtained by the functional regression of $Y$ on $W(\cdot) $ as proposed by \cite{muller2005}. Let $\hat{\beta}^{naive}_k=\int \hat{\beta}^{naive}(t)\hat{\rho}_k(t)dt. $ We denote the basis functions formed by the eigenfunctions of the integral operator associated with the estimated covariance function $\hat{K}(\cdot,\cdot)$ as $\hat{\rho}_k(\cdot),\,k\geq1$. Then, $\hat{\beta}^{naive} =(\hat{\beta}^{naive}_1,...,\hat{\beta}^{naive}_{p_n}) $ is used as the initial value. The solution to this algorithm is the corrected estimate denoted by $ \hat{\beta}^{cor}=(\hat{\beta}^{cor}_1,...,\hat{\beta}^{cor}_{p_n}) $. The corrected estimate $\hat{\beta}^{cor}(\cdot)$ is $\hat{\beta}^{cor} (\cdot)=\sum\limits_{k=1}^{p_n} \hat{\rho}_k(\cdot) \hat{\beta}^{cor}_k $. The estimation errors are calculated as $ E_n=\int (\tilde{\beta}(t)-\hat{\beta}^{naive}(t)) ^2dt$ and  $ E_{co}=\int (\tilde{\beta}(t)-\hat{\beta}^{cor}(t)) ^2dt$. 

In the Gaussian case, we perform comparison with the PACE method proposed by \cite{yao2005functional} and the calibration based estimate proposed by \cite{chakraborty2017regression}. PACE was implemented using MATLAB packages. Covariance estimation for calibration method was implemented with the help of the code provided by the authors. The error corresponding to PACE and calibration estimates are denoted by $E_p$ and $E_{ca}$ respectively. Tables 1-3 report average errors based on $200$ repetitions. The values reported in the column $p_n$ denote the averages of the number of components selected.

\begin{table}[t!] 
	\caption{Mean Error for Gaussian Response in Setting 1.}
	\label{table1}\par
	\vskip .2cm
	\centerline{\tabcolsep=4truept\begin{tabular}{|c|ccc|ccc|ccc|} \hline 
			\multicolumn{1}{|c}{Sample}
			& \multicolumn{3}{c}{$n=1000$}
			& \multicolumn{3}{c}{$n=3000$}
			& \multicolumn{3}{c|}{$n=6000$} \\ \hline
			\multicolumn{1}{|c}{$l$}
			& \multicolumn{1}{c}{$0.050$} & \multicolumn{1}{c}{$0.080$} & \multicolumn{1}{c}{$0.100$}
			& \multicolumn{1}{c}{$0.050$} & \multicolumn{1}{c}{$0.080$} & \multicolumn{1}{c}{$0.100$}
			& \multicolumn{1}{c}{$0.050$} & \multicolumn{1}{c}{$0.080$} & \multicolumn{1}{c|}{$0.100$} \\[3pt] \hline
		    $p_n$ &8 &8 &8 &10 &10 &10 &12 &12 &12\\
			  $E_n$  & 0.090 &0.158 & 0.207&0.085 & 0.160 & 0.202&0.084 &0.156&0.202 \\
		  	 \bf{$E_{co}$} &\textit{0.027} &\textit{0.052} &\textit{0.071} &\textit{0.024} &\textit{0.047} &\textit{0.060} &\textit{0.024} &\textit{0.045} & \textit{0.056}  \\
			 $E_p$ &1.600 & 8.140 &18.41 &0.991&6.861&2.333&0.656&1.573&2.956\\
			 $E_{ca}$ &0.093 &0.170 &0.220 &0.091 &0.171 &0.216 &0.092 &0.167 & 0.217 
			\\ \hline
			\end{tabular}}
\end{table}

\begin{table}[t!] 
	\caption{Mean Error for Gaussian Response in Setting 2.}
	\label{table2}\par
	\vskip .2cm
	\centerline{\tabcolsep=2truept\begin{tabular}{|c|ccc|ccc|ccc|} \hline 
			\multicolumn{1}{|c}{Sample}
			& \multicolumn{3}{c}{$n=1000$}
			& \multicolumn{3}{c}{$n=3000$}
			& \multicolumn{3}{c|}{$n=6000$} \\ \hline
			\multicolumn{1}{|c}{$\sigma_2$}
			& \multicolumn{1}{c}{$1.00$} & \multicolumn{1}{c}{$2.00$} & \multicolumn{1}{c}{$3.00$}
			& \multicolumn{1}{c}{$1.00$} & \multicolumn{1}{c}{$2.00$} & \multicolumn{1}{c}{$3.00$}
			& \multicolumn{1}{c}{$1.00$} & \multicolumn{1}{c}{$2.00$} & \multicolumn{1}{c|}{$3.00$} \\[3pt] \hline
			$p_n$ &8 &8 &8 &10 &10 &10 &12 &12 &12\\
			$E_n$  & 0.006   &0.013 &0.033 &0.003  &0.010 &0.030 &0.003 &0.009  & 0.031  \\
			\textit{$E_{co}$} &\textit{0.004 } &\textit{0.004 } &\textit{0.006 } &\textit{0.001 } &\textit{0.002 } &\textit{0.002 } &\textit{0.001 } &\textit{0.001 } &\textit{0.001 }  \\
			$E_p$ &37.44  &31.95   &26.71  & 0.005 &2.477 &3.591 &0.005 &0.012  &0.031  \\
			$E_{ca}$ & 0.009  & 0.017 &0.040  &0.006 &0.014 &0.037  &0.005  &0.014  &0.036  
			\\ \hline
	\end{tabular}}
\end{table}

\begin{table}[t!] 
	\caption{Mean Error for Binary Response in Setting 1.}
	\label{table3}\par
	\vskip .2cm
	\centerline{\tabcolsep=4truept\begin{tabular}{cccccccccc} \hline 
			\multicolumn{1}{c}{Sample}
			& \multicolumn{3}{c}{$n=1000$}
			& \multicolumn{3}{c}{$n=3000$}
			& \multicolumn{3}{c}{$n=6000$} \\ \hline
			\multicolumn{1}{c}{$l$}
			& \multicolumn{1}{c}{$0.050$} & \multicolumn{1}{c}{$0.080$} & \multicolumn{1}{c}{$0.100$}
			& \multicolumn{1}{c}{$0.050$} & \multicolumn{1}{c}{$0.080$} & \multicolumn{1}{c}{$0.100$}
			& \multicolumn{1}{c}{$0.050$} & \multicolumn{1}{c}{$0.080$} & \multicolumn{1}{c}{$0.100$} \\[3pt] \hline
			$E_n$  &1.209  & 1.151 &1.110 &1.587  &1.579  &1.573  &1.581 &1.575 &1.572  \\
			\bf{$E_{co}$} &\textit{0.153 } &\textit{0.151 } &\textit{0.167 } &\textit{0.082} &\textit{0.121} &\textit{0.120} &\textit{0.071 } & \textit{0.100 } &\textit{0.110} 	\\ \hline
	\end{tabular}}
\end{table}

\begin{table}[t!] 
	\caption{Mean Error for Binary Response in Setting 2.}
	\label{table4}\par
	\vskip .2cm
	\centerline{\tabcolsep=4truept\begin{tabular}{cccccccccc} \hline 
			\multicolumn{1}{c}{Sample}
			& \multicolumn{3}{c}{$n=1000$}
			& \multicolumn{3}{c}{$n=3000$}
			& \multicolumn{3}{c}{$n=6000$} \\ \hline
			\multicolumn{1}{c}{$\sigma_2$}
			& \multicolumn{1}{c}{$1.00$} & \multicolumn{1}{c}{$2.00$} & \multicolumn{1}{c}{$3.00$}
			& \multicolumn{1}{c}{$1.00$} & \multicolumn{1}{c}{$2.00$} & \multicolumn{1}{c}{$3.00$}
			& \multicolumn{1}{c}{$1.00$} & \multicolumn{1}{c}{$2.00$} & \multicolumn{1}{c}{$3.00$} \\[3pt] \hline
			$E_n$  & 1.314   & 1.301  &1.230   &1.587    &1.592   &1.589  &1.587  &1.587   &1.583    \\
			\textit{$E_{co}$} &\textit{0.184 } &\textit{0.190 } &\textit{0.208 } &\textit{0.072 } &\textit{0.085 } &\textit{0.101} &\textit{0.040 } &\textit{0.061 } &\textit{ 0.070}  
			\\ \hline
	\end{tabular}}
\end{table}

%

From Tables 1--4, we observe that the error of the proposed corrected estimate, highlighted in italics, is lower than all the other alternatives. From the Table \ref{table1}, we observe that the error of the PACE estimator increases with $l$. PACE assumes that measurement error occurs at only discrete realizations of the function. The calibration estimator does allow the error to be functional in nature but requires the covariance structure to exist on a very small interval. That is, it imposes a banded structure on the covariance function of the error process. This explains the deteriorating performance of both these methods with an increasing value of $l$. The naive estimator ignores measurement error, and accordingly it's performance decreases with increasing $l$. A similar trend observed in the performance of the proposed corrected estimate, can be explained by the fact that the naive estimate is used as the initial value in the Newton Raphson algorithm. From Table \ref{table2}, we observe that the performance of all the methods deteriorate as the noise ($\sigma_2$) increases. Table \ref{table3} and \ref{table4} exhibit similar trends.

\subsection{Canadian weather data}
We now perform a simulation study based on the Canadian weather data.  This data consists of daily temperature measurements obtained from 35 Canadian weather stations for a period of one year. It also contains total annual rainfall, on a log scale at each of these stations. A sample of three curves from are displayed in Figure 2b. Denote the log annual  precipitation and curves obtained from smoothing the temperature data by $(Y_i,X_i(\cdot)),\, i=1,...,35$. Details on the smoothing method can be found in \cite{ramsay2006functional}. With the development in technologies and the low variability in the data, it is reasonable to assume that this data is free from measurement error. That is, the observed data is indeed the true covariate and error is added to it to obtain the contaminated covariate $W_i(\cdot)=X_i(\cdot)+U_i(\cdot)$, where $U_i(\cdot),\, i=1,...n$ are indepedent and identically distributed as a centered Gaussian Process with the squared exponential covariance function $K(s,t)=5\text{exp}(-(s-t)^2/(2*0.5^2))$. The slope function $\tilde{\beta}(\cdot)$ referred to as the true slope, is obtained via the regression model $Y_i=\int X(t)\beta(t)dt+\varepsilon_i$. It is assumed that the temperature has a linear effect on the log precipitation, there are no other covariates present in the model and that the slope function is indeed accurate. 
From Figure 2a, we can see that the corrected estimate is closer to $ \tilde{\beta}(\cdot)$ than the naive one. The errors are $E_n=0.52,\, E_{co}=0.01$. Thus, in the presence of measurement error, the corrected estimate offers a marked improvement in the estimates. 

\begin{figure}[h!]
	\centerline{\epsfig{file=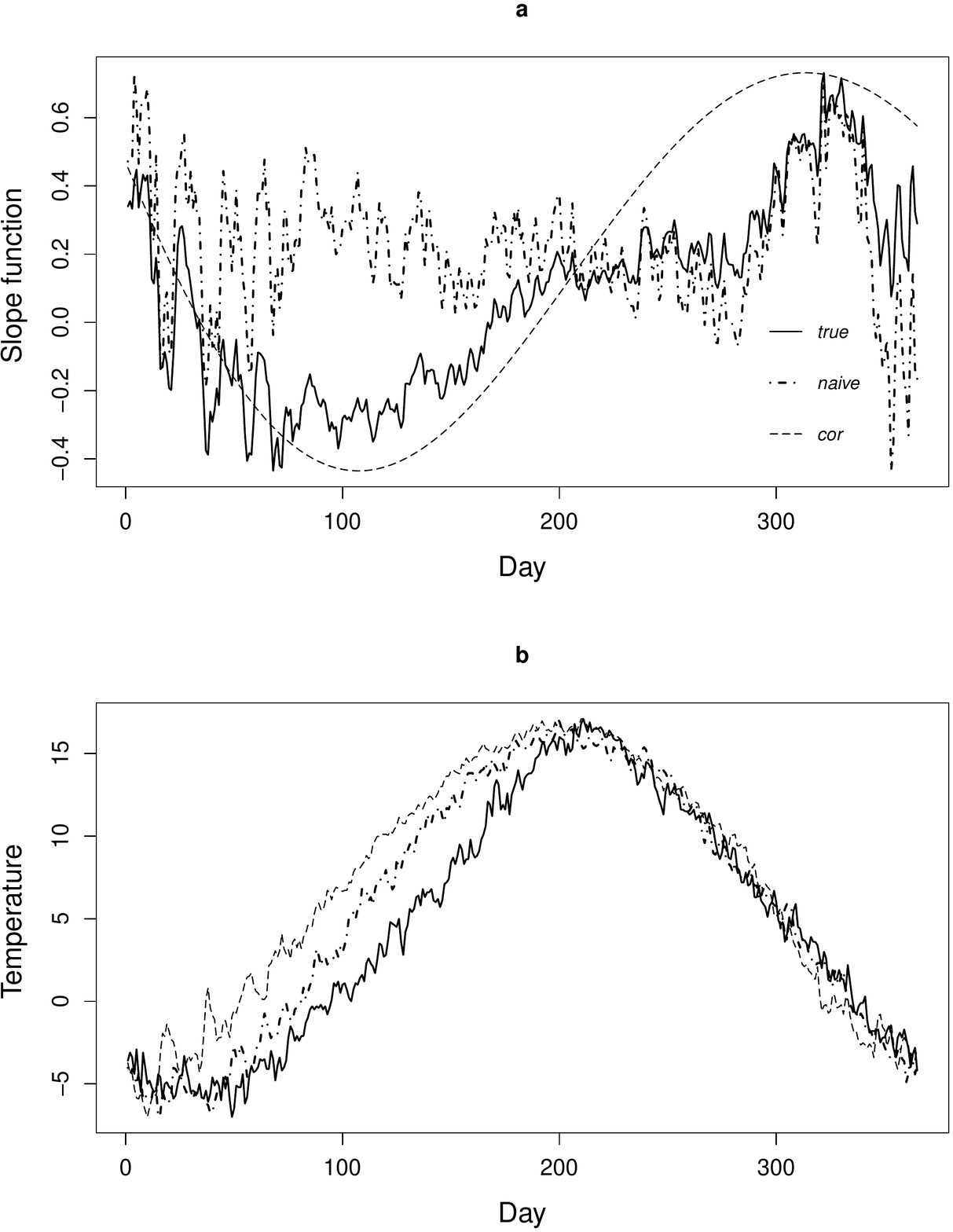,width=4.5in}}\par  
	\caption{a: Plots of the true, naive and corrected slope function estimate. b: Plots of temperature at 3 of the stations.}
\end{figure}

\section{Conclusion}
We propose a mechanism to account for functional measurement error in functional regression models. Moreover, it is the first attempt in which the measurement error is functional and the response is binary. It allows us to relax several assumptions made on the covariance structure of measurement error in the presence of replicates. This can also serve as a guideline for collecting additional data when there is a possibility of data contamination. Simulations clearly indicate the competitive performance of our method over several alternatives, especially in the Gaussian case. 
There is a wide scope for future work in this area of measurement error in Functional Data Analysis. Most existing methods including the one proposed here, assume the classical linear measurement error model and involve parametric regression models. Models can be investigated that allow for different measurement error models particularly in the non-parametric regression framework.

%



\bibhang=1.7pc
\bibsep=2pt
\fontsize{9}{14pt plus.8pt minus .6pt}\selectfont
\renewcommand\bibname{\large \bf References}
\expandafter\ifx\csname
natexlab\endcsname\relax\def\natexlab#1{#1}\fi
\expandafter\ifx\csname url\endcsname\relax
  \def\url#1{\texttt{#1}}\fi
\expandafter\ifx\csname urlprefix\endcsname\relax\def\urlprefix{URL}\fi

\section{Appendix }
\subsection{Appendix A }

\begin{proposition} \label{score}
	Let $\tilde{\beta}_c=(\tilde{\beta_0},...,\tilde{\beta}_{p_n})^{'}$. Then,
	$$\| U(\tilde{\beta}_c)\|=O_p(\sqrt{n}).$$
	
\end{proposition}

\begin{proof}
	Let $\lambda_0=0$. Then,	
	\begin{align}
	U(\tilde{\beta}_c)&= \sum_{i=1}^{n} (\{Y_i-E(Y_i|\delta_{ci})\}(\delta_{ci}-\Omega_c\tilde{\beta}_c) \nonumber\\
	\|U(\tilde{\beta}_c)\|^2&=\sum_{l=0}^{p_n} \left\{\sum_{i=1}^{n} 
	(Y_i-E(Y_i|\delta_{ci}))(\delta_{ci}^{(l)}- \lambda_{l}\tilde{\beta}_{l})\right\}^2  \nonumber \\
	= & \sum_{l=0}^{p_n} \sum_{i=1}^{n} \left[ \{Y_i-E(Y_i|\delta_{ci})\}(\delta_{ci}^{(l)}- \lambda_{l}\tilde{\beta}_{l}) \right]^2  \nonumber\\
	& +\sum_{l=0}^{p_n} \sum_{i_1\neq i_2=1}^{n}\{Y_{i_1}-E(Y_{i_1}|\delta_{ci_1})\}(\delta_{ci_1}^{(l)}- \lambda_{l}\tilde{\beta}_{l}) (Y_{i_2}-E(Y_{i_2}|\delta_{ci_2}))(\delta_{ci_2}^{(l)}- \lambda_{l}\tilde{\beta}_{l})  \nonumber\\
	&=A_1+A_2, \nonumber\\
A_1&=   \sum_{l=0}^{p_n} \sum_{i=1}^{n} \left\{ (Y_i-E(Y_i|\delta_{ci}))(\delta_{ci}^{(l)}- \lambda_{l}\tilde{\beta}_{l}) \right\}^2  \nonumber \\
	A_1&= \sum_{i=1}^{n} (Y_i-E(Y_i|\delta_{ci}))^2\sum_{l=0}^{p_n}(\delta_{ci}^{(l)}- \lambda_{l}\tilde{\beta}_{l})^2  \nonumber\\
	A_1 &=\sum_{i=1}^{n} \{Y_i-E(Y_i|\delta_{ci})\}^2 \|\delta_{ci}-\Omega_c\tilde{\beta}_c\|^2 . \nonumber
	\end{align} 
	From Assumption \ref{beta},\, \ref{x},\,\ref{omega},
	\begin{align*} 
	\|\delta_{ci}-\Omega_c\tilde{\beta}_c\|^2 &= \|W_{ci}+Y_i\Omega_c \tilde{\beta}_c -  \Omega_c\tilde{\beta}_c\|^2 \\
	& \leq \|X_{ci}\|^2+ \|U_{ic}\|^2+ \|\Omega_c\|^2\|\tilde{\beta}_c\|^2  \leq c+\sum_{l=1}^{p_n} U_{il}^2.
	\end{align*}
	Recall that $u_{il} \sim N(0,\lambda_l)$. Thus,
	\begin{align*}
	E(A_1) & \leq c \sum_{i=1}^{n} E[\{Y_i-E(Y_i|\delta_{ci})\}^2]\sum_{l=1}^{p_n} EU_{il}^2\\
	&	\leq \sum_{i=1}^{n} E[\{Y_i-E(Y_i|\delta_{ci})\}^2]\sum_{l=1}^{p_n}\lambda_l =O(n). 
	\end{align*} 	 
	We now examine $A_2$ which is given as 	 
	\begin{align}	
	&\sum_{l=0}^{p_n} \sum_{i_1\neq i_2=1}^{n}\{Y_{i_1}-E(Y_{i_1}|\delta_{ci_1})\}(\delta_{ci_1}^{(l)}- \lambda_{l}\beta_{l}) (Y_{i_2}-E(Y_{i_2}|\delta_{ci_2}))(\delta_{ci_2}^{(l)}-\lambda_{l}\beta_{l})  \nonumber\\
	&=\sum_{i_1\neq i_2=1}^{n} \{Y_{i_1}-E(Y_{i_1}|\delta_{ci_1})\}(Y_{i_2}-E(Y_{i_2}|\delta_{ci_2})) \sum_{l=0}^{p_n} (\delta_{ci_1}^{(l)}- \lambda_{l}\beta_{l})(\delta_{ci_2}^{(l)}- \lambda_{l}\beta_{l}).  \nonumber 	\end{align}
	Thus, we obtain $E(A_2) = 0$ and $E(\|U(\beta_c)\|^2)=O(n)$ implying $ \|U(\beta_c)\|=O_p(\sqrt{n}). $	
	
\end{proof}

\begin{proposition} \label{jacobian}
	Let $J(\beta_c)=\dfrac{\partial U(\beta_c)}{\partial \beta_c}$ and $t_i(\beta_c)=W_{ci}^{'}\beta_c+(Y_i-0.5)\beta_c^{'}\Omega\beta_c$. Then 
	\begin{align*}
	J(\beta_c)&=-\sum_{i=1}^{n}  F(t_i(\beta_c))(Y_i-1)\Omega_c-\sum_{i=1}^{n}  F(t_i(\beta_c))(1-F(t_i(\beta_c)))(W_{ic}W_{ic}^{'}+(Y_i-1)\Omega_c\beta_cW_{ic}^{'})\\
	&\quad-\sum_{i=1}^{n}  F(t_i(\beta_c))(1-F(t_i(\beta_c)))(2(Y_i-0.5)W_{ic}(\Omega_c\beta_c)^{'}
	+2(Y_i-1)(Y_i-0.5)\Omega_c\beta_c(\Omega_c\beta_c)^{'})
	\end{align*} 
	
\end{proposition}

\begin{proof}
	\begin{align*}
	U(\beta_c)&= \sum_{i=1}^{n}(Y_i-F(\delta_{ci}^{'}\beta_c-0.5\beta_c^{'}\Omega_c\beta_c))\begin{bmatrix}
	\delta_{ci}-\Omega_c\beta_c\end{bmatrix} \\
	&= \sum_{i=1}^{n}(Y_i-F(W_{ci}^{'}\beta_c+(Y_i-0.5)\beta_c^{'}\Omega_c\beta_c))\begin{bmatrix}
	W_{ci}+(Y_i-1)\Omega_c\beta_c\end{bmatrix}\\
	&= \sum_{i=1}^{n}Y_i	W_{ci} +Y_i(Y_i-1)\Omega_c\beta_c- F(W_{ci}^{'}\beta_c+(Y_i-0.5)\beta_c^{'}\Omega_c\beta_c)\begin{bmatrix}
	W_{ci}+(Y_i-1)\Omega_c\beta_c\end{bmatrix}
	\end{align*}
	Taking derivatives
	\begin{align*}
	\dfrac{\partial U(\beta_c) }{\partial \beta_c} &= -\sum_{i=1}^{n} F(t_i(\beta_c))(Y_i-1)\Omega_c-\sum_{i=1}^{n}  F(t_i(\beta_c))(1-F(t_i(\beta_c)))(W_{ic}W_{ic}^{'}+(Y_i-1)\Omega_c\beta_cW_{ic}^{'})\\
	&\quad -\sum_{i=1}^{n}  F(t_i(\beta_c))(1-F(t_i(\beta_c)))(2(Y_i-0.5)W_{ic}(\Omega_c\beta_c)^{'}+
	2(Y_i-1)(Y_i-0.5)\Omega_c\beta_c(\Omega_c\beta_c)^{'})
	\end{align*}	
\end{proof}

\begin{proposition} \label{prop3}
	$\underset{\| \beta_c-\tilde{\beta}_c\|= \zeta \sqrt{p_n/n}}{\sup}
	\underset{\|b\|=1}{\sup} | b^{'} (J(\beta_c)-J(\tilde{\beta}_c ))b | \leq \sqrt{n} p_n$
	
\end{proposition}

\begin{proof}
	Consider,
	\begin{align*}
	|b^{'}  \sum_{i=1}^{n} (F(t_i(\beta_c))-F(t_i(\tilde{\beta}_c)))(y_i-1)\Omega_c b |\leq \sum_{i=1}^{n} |(F(t_i(\beta_c))-F(t_i(\tilde{\beta}_c))) ||b^{'} \Omega_c b|
	\end{align*}	
	Using the fact that $F(\cdot)$ is a bounded continuous function with bounded derivatives along with Taylors integral remainder theorem we obtain
	
	$$| F(t_i(\beta_c))-F(t_i(\tilde{\beta}_c)) |< c |t_i(\beta_c)-t_i(\tilde{\beta}_c)| \leq \|W_{ci}\|\|\beta_c-\tilde{\beta}_c\|+0.5\|\beta_c^{'}\Omega_c \beta_c^{'}-\tilde{\beta}_c^{'}\Omega_c\tilde{\beta}_c^{'} \|.$$
	Thus from Assumptions \ref{beta},\,\ref{x} and \ref{omega},
	
	\begin{align}
	& |	b^{'}  \sum_{i=1}^{n} (F(t_i(\beta_c))-F(t_i(\tilde{\beta}_c)))(y_i-1)\Omega_c b | \nonumber \\
	&\leq  \sum_{i=1}^{n} \left(\|W_{ci}\|\|\beta_c-\tilde{\beta}_c\|+0.5\|\beta_c^{'}\Omega_c \beta_c^{'}-\tilde{\beta}_c^{'}\Omega_c\tilde{\beta}_c^{'} \| \right) |b^{'}\Omega_c b| \nonumber\\	
	& \leq |b^{'}\Omega_c b| \|\beta_c-\tilde{\beta}_c\| \sum_{i=1}^{n} \left( \|W_{ci}\|+ c \|\Omega\| \|\tilde{\beta}_c \| \right)\nonumber . 
	\end{align}
	Taking supremum we obtain,
	\begin{align}
	&	\underset{\| \beta_c-\tilde{\beta}_c\|= \zeta \sqrt{p_n/n}}{\sup}
	\underset{\|b\|=1}{\sup}b^{'} | \sum_{i=1}^{n} (F(t_i(\beta_c))-F(t_i(\tilde{\beta}_c)))(y_i-1)\Omega_c b| \leq O_p(\sqrt{n}p_n) .\label{prop2_1}
	\end{align}
	Consider,
	$$ |\sum_{i=1}^{n} ( F(t_i(\beta_c))(1-F(t_i(\beta_c)))-F(t_i(\tilde{\beta}_c))(1-F(t_i(\tilde{\beta}_c))))b^{'}W_{ic}W_{ic}^{'} b | $$
	By using the boundedness of $ F(t_i(\beta_c))(1-F(t_i(\beta_c)))$ we obtain
	\begin{align*}
	&|\sum_{i=1}^{n} \left( F(t_i(\beta_c))(1-F(t_i(\beta_c)))-F(t_i(\tilde{\beta}_c))(1-F(t_i(\tilde{\beta}_c)))\right)b^{'}W_{ic}W_{ic}^{'} b|^2 \\
	&= n^2 \left(\sum_{i=1}^{n}\dfrac{  F(t_i(\beta_c))(1-F(t_i(\beta_c)))-F(t_i(\tilde{\beta}_c))(1-F(t_i(\tilde{\beta}_c))) }{\sqrt{n}} \dfrac{b^{'}X_{ic}X_{ic}^{'} b+b^{'}U_{ic}U_{ic}^{'} b}{\sqrt{n}}\right)^2\\
	& \leq 4cn^2 \sum_{i=1}^{n}\dfrac{ (\|W_{ci}\|^2\|\beta_c-\tilde{\beta}_c\|^2+0.5\|\beta_c^{'}\Omega_c \beta_c^{'}-\tilde{\beta}_c^{'}\Omega_c\tilde{\beta}_c^{'} \|^2) }{n}  \sum_{i=1}^{n}\dfrac{(b^{'}X_{ic}X_{ic}^{'} b)^2+(b^{'}U_{ic}U_{ic}^{'} b)^2}{n}\\
	&\leq 4cn^2 \sum_{i=1}^{n}\dfrac{ (\|W_{ci}\|^2\|\beta_c-\tilde{\beta}_c\|^2+0.5\|\beta_c^{'}\Omega_c \beta_c^{'}-\tilde{\beta}_c^{'}\Omega_c\tilde{\beta}_c^{'} \|^2) }{n}   \left(\sum_{i=1}^{n}\dfrac{(b^{'}X_{ic}X_{ic}^{'} b)^2}{n}+E(b^{'}U_{1c}U_{1c}^{'} b)^2 \right)\\
	&+\text{smaller order terms} \\
	\end{align*}
	Using Assumption \ref{omega}, 
	\begin{align*}
	&	E(b^{'}u_{1c}u_{1c}^{'} b)^2=E(\|b^{'}U_{1c}\|^4) \leq \|b\|^4 E(\|U_{1c}\|^4)\\
	&E(\|U_{1c}\|^4)	= E(\sum_{l=1}^{p_n} U_{1cl}^2)^2= \sum_{l=1}^{p_n} EU_{1cl}^4+\sum_{l=1}^{p_n} \lambda_{l}\sum_{l=1}^{p_n} \lambda_{l} =O(p_n)
	\end{align*}
	Thus,
	
	\begin{align*}
	&	|\sum_{i=1}^{n} \left( F(t_i(\beta_c))(1-F(t_i(\beta_c)))-F(t_i(\tilde{\beta}_c))(1-F(t_i(\tilde{\beta}_c)))\right)b^{'}W_{ic}W_{ic}^{'} b|^2 \\
	&\leq\|\beta_c-\tilde{\beta}_c\|^2 n^2 O_p(p_n)
	\end{align*}
	Taking supremum,
	\begin{align*}
	&	\underset{\| \beta_c-\tilde{\beta}_c\|= \zeta \sqrt{p_n/n}}{\sup}
	\underset{\|b\|=1}{\sup}b^{'}  	|\sum_{i=1}^{n} \left( F(t_i(\beta_c))(1-F(t_i(\beta_c)))-F(t_i(\tilde{\beta}_c))(1-F(t_i(\tilde{\beta}_c)))\right)b^{'}W_{ic}W_{ic}^{'} b| \\
	&  \leq   \zeta O_p(\dfrac{\sqrt{p_n}}{\sqrt{n}}n \sqrt{p_n})  =O_p(\sqrt{n}p_n).
	\end{align*}

	Similarly we can show that the other involved terms are of order $\sqrt{n}p_n$ thus yielding the result.
\end{proof}

\begin{proposition}\label{prop4}
	$$\underset{\| \beta_c-\tilde{\beta}_c\|= \zeta \sqrt{p_n/n}}{sup} (\beta_c-\tilde{\beta}_c)^{'}J(\tilde{\beta}_c)(\beta_c-\tilde{\beta}_c)<-c\zeta^2 p_n+ \zeta^2o(p_n) $$
\end{proposition}

\begin{proof}
	Let $a_c=\beta_c-\tilde{\beta}_c,\,a=\beta-\tilde{\beta} $
	\begin{align*}
	a_c^{'}J(\tilde{\beta}_c)a&=a_c^{'} \sum_{i=1}^{n}  F(t_i(\tilde{\beta}_c))(1-Y_i)\Omega_c a_c 	\\
	& - a_c^{'}\sum_{i=1}^{n}  F(t_i(\tilde{\beta}_c))(1-F(t_i)(\tilde{\beta}_c))W_{ic}W_{ic}^{'})a_c\\
	&- a_c^{'}\sum_{i=1}^{n}F(t_i(\tilde{\beta}_c))(1-F(t_i(\tilde{\beta}_c)))(Y_i-1)\Omega_c\beta_cW_{ic}^{'}a_c \\
	&-a_c^{'}\sum_{i=1}^{n}F(t_i(\tilde{\beta}_c))(1-F(t_i(\tilde{\beta}_c)))(Y_i-0.5)2W_{ic}(\Omega_c\beta_c)^{'}a_c \\
	&-a_c\sum_{i=1}^{n}F(t_i(\tilde{\beta}_c))(1-F(t_i(\tilde{\beta}_c))) 2(Y_i-1)(Y_i-0.5)\Omega_c\beta_c(\Omega_c\beta_c)^{'}a_c\\
	&=A_1+A_2+A_3+A_4+A_5
	\end{align*}	
	
	We study the orders of each of these terms. Consider,
	\begin{align}
	A_3&=\sum_{i=1}^{n}F(t_i(\tilde{\beta}_c))(1-F(t_i(\tilde{\beta}_c)))(1-Y_i)a_c^{'}\Omega_c\beta_cW_{ic}^{'}a_c \nonumber \\
	&\leq \underset{ i} {sup} F(t_i(\tilde{\beta}_c))(1-F(t_i(\tilde{\beta}_c))) a^{'}\Omega \tilde{\beta} n\sum_{i=1}^{n}\dfrac{X_{i}^{'}+U_{i}^{'}}{n} a \nonumber \\
	&= \zeta^2 o_p(p_n)  \label{prop4_1}
	\end{align}
	In a similar way we can show that $A_4 = \zeta^2 o_p(p_n)$.
	Consider,
	\begin{align}
	A_1=a_c^{'} \sum_{i=1}^{n}  F(t_i(\tilde{\beta}_c))(1-Y_i)\Omega_c a_c & \leq 
	\underset{1\leq i \leq n}{sup}F(t_i(\tilde{\beta}_c)) a_c^{'}\Omega_c a_c n \sum_{i=1}^{n} (1-Y_i)/n \nonumber \\
	&  \leq 
	\underset{1\leq i \leq n}{sup}F(t_i(\tilde{\beta}_c)) a_c^{'}\Omega_c a_c n \sum_{i=1}^{n} (1-F(X_{ic}^{'}\tilde{\beta}_c))/n \nonumber \\
	& \quad + \text{smaller order term}\nonumber\\
	& \leq 
	\underset{1\leq i \leq n}{sup}F(t_i(\tilde{\beta}_c)) (1-\underset{1\leq i \leq n}{inf}F(X_{ic}^{'}\tilde{\beta}_c)) a_c^{'}\Omega_c a_c n \nonumber             
	\end{align}
	Note that the function $F(t)=1/(1+e^{-t})$ is a non-decreasing function and  $\Omega_c$ is positive semidefinite. Using $$ \underset{1\leq i \leq n}{sup } t_i(\tilde{\beta}_c) =\underset{1\leq i \leq n}{sup} W_{ci}^{'}\tilde{\beta}_c+(Y_i-0.5)\tilde{\beta}_c^{'}\Omega\tilde{\beta}_c \leq \underset{1\leq i \leq n}{sup} W_{ci}^{'}\tilde{\beta}_c+0.5\tilde{\beta}_c^{'}\Omega\tilde{\beta}_c=t_1, $$
	$$ \underset{1\leq i \leq n}{inf } t_i(\tilde{\beta}_c) =\underset{1\leq i \leq n}{inf} W_{ci}^{'}\tilde{\beta}_c+(Y_i-0.5)\tilde{\beta}_c^{'}\Omega\tilde{\beta}_c \geq \underset{1\leq i \leq n}{inf} W_{ci}^{'}\tilde{\beta}_c-0.5\tilde{\beta}_c^{'}\Omega\tilde{\beta}_c=t_2, $$
	
	we obtain, $\underset{1\leq i\leq n}{sup} F(t_i(\tilde{\beta}_c)) \leq F\left(\underset{1\leq i\leq n}{sup} t_i(\tilde{\beta}_c)\right)\leq \dfrac{e^{t_1}}{1+e^{t_1}} $ and
	$\dfrac{e^{t_2}}{1+e^{t_2}}=  F(t_2)  \leq \underset{1\leq i\leq n}{inf} F(t_i(\tilde{\beta}_c))  .$ 
	
	Let $t_3= \underset{1\leq i\leq n}{inf} X_{ic}^{'}\tilde{\beta}_c,\,t_4= \underset{1\leq i\leq n}{sup} X_{ic}^{'}\tilde{\beta}_c   $.
	\begin{align*}
	A_1=a_c^{'} \sum_{i=1}^{n}  F(t_i(\tilde{\beta}_c))(1-Y_i)\Omega_c a_c & \leq     \dfrac{e^{t_1}}{(1+e^{t_1})(1+e^{t_3})} \lambda_{max}(\Omega)\zeta^2 p_n 
	\end{align*}


	\begin{align}
	A_4&=-a_c\sum_{i=1}^{n}F(t_i(\tilde{\beta}_c))(1-F(t_i(\tilde{\beta}_c))) 2(Y_i-1)(Y_i-0.5)\Omega_c\beta_c(\Omega_c\beta_c)^{'}a_c \nonumber \\
	&<  0  \label{prop4_2}
	\end{align}

	\begin{align}
	A_2 &= - a_c^{'}\sum_{i=1}^{n}  F(t_i(\tilde{\beta}_c))(1-F(t_i)(\tilde{\beta}_c))(W_{ic}W_{ic}^{'})a_c \nonumber \\
	& \leq - \dfrac{e^{t_2}}{(1+e^{t_2})(1+e^{t_1})} n a^{'}\lambda_{min}\left(\sum_{i=1}^{n} \dfrac{W_{ic}W_{ic}^{'}}{n}\right)a \nonumber \\
	&\quad +\text{smaller order term} \nonumber \\
	A_1+A_2 & \leq  \dfrac{e^{t_1}}{(1+e^{t_1})(1+e^{t_3})} \lambda_{max}(\Omega)\zeta^2 p_n \nonumber \\ &-\dfrac{e^{t_2}}{(1+e^{t_2})(1+e^{t_1})}  \lambda_{min}\left(\sum_{i=1}^{n} \dfrac{W_{ic}W_{ic}^{'}}{n}\right)\zeta^2 p_n \nonumber \\
	&= \dfrac{\zeta^2 p_n}{1+e^{t_1}}\left(\dfrac{e^{t_1}}{1+e^{t_3}} \lambda_{max}(\Omega) -  \dfrac{e^{t_2}}{1+e^{t_2}}  \lambda_{min}\left(\sum_{i=1}^{n} \dfrac{W_{ic}W_{ic}^{'}}{n}\right)\right) \nonumber   
	\end{align}
	Let $\lambda_{min}\left(\sum_{i=1}^{n} \dfrac{W_{ic}W_{ic}^{'}}{n}\right)=B $.  From Assumption \ref{eigen} 
	\begin{align*}
	\lambda_{max}(\Omega)\text{exp}(\lambda_{max}(\Omega)m) & \leq B \dfrac{\text{exp}(inf W_{ic}^{'}\tilde{\beta}-sup W_{ic}^{'}\tilde{\beta})}{1+\text{exp}(inf W_{ic}^{'}\tilde{\beta})} \\
	\lambda_{max}(\Omega) & \leq  B \dfrac{\text{exp}(inf W_{ic}^{'}\tilde{\beta}-sup w_{ic}^{'}\tilde{\beta}-\lambda_{max}(\Omega)\| \tilde{\beta}\|^2)}{1+\text{exp}(inf W_{ic}^{'}\tilde{\beta})}  \\
	& \leq  B \dfrac{\text{exp}(inf W_{ic}^{'}\tilde{\beta}-sup W_{ic}^{'}\tilde{\beta}-\lambda_{max}(\Omega)\| \tilde{\beta}\|^2)}{1+\text{exp}(inf W_{ic}^{'}\tilde{\beta}-0.5\lambda_{min}(\Omega) \| \tilde{\beta}\|^2 )}\\
	& \leq  B \dfrac{\text{exp}(inf W_{ic}^{'}\tilde{\beta}-sup W_{ic}^{'}\tilde{\beta}-\tilde{\beta}^{'}\Omega \tilde{\beta})}{1+\text{exp}(inf W_{ic}^{'}\tilde{\beta}-0.5\tilde{\beta}^{'}\Omega \tilde{\beta}) }\\
	& \leq B  \dfrac{\text{exp}(inf W_{ic}^{'}\tilde{\beta}-0.5\tilde{\beta}^{'}\Omega \tilde{\beta})}
	{1+\text{exp}(inf W_{ic}^{'}\tilde{\beta}-0.5\tilde{\beta}^{'}\Omega \tilde{\beta})} \dfrac{1+\text{exp}( inf X_{ic}^{'} \tilde{\beta})} {\text{exp}(sup W_{ic}^{'}\tilde{\beta}+0.5\tilde{\beta}^{'}\Omega \tilde{\beta})}\\
	\lambda_{max}(\Omega) \dfrac{\text{exp}(sup W_{ic}^{'}\tilde{\beta}+0.5\tilde{\beta}^{'}\Omega \tilde{\beta})}{1+\text{exp}( inf W_{ic}^{'} \tilde{\beta})} & \leq  B  \dfrac{\text{exp}(inf W_{ic}^{'}\tilde{\beta}-0.5\tilde{\beta}^{'}\Omega \tilde{\beta})}
	{1+\text{exp}(inf W_{ic}^{'}\tilde{\beta}-0.5\tilde{\beta}^{'}\Omega \tilde{\beta})}     \\
	\end{align*}
	Thus, $A_1+A_2 \leq -c\zeta^2p_n$
	and along with   \eqref{prop4_1},\eqref{prop4_2} we get the result.
	
\end{proof}	

\subsection{Appendix B }

\begin{proposition} \label{c_prop1}
	$$ \| U(\tilde{\beta})\|=O_p(\sqrt{n}).$$
\end{proposition}

\begin{proof}
	\begin{align*}
	\|U(\tilde{\beta})\|&\leq \|\sum\limits_{i=1}^{n} \mathbf{W}_i^cY_i^c\tilde{\beta}^{'}\Omega\tilde{\beta} \| +\| \sum\limits_{i=1}^{n}  Y^{c2}_i\Omega\tilde{\beta}\|+\|\sum\limits_{i=1}^{n}\mathbf{W}_i^c\mathbf{W}_i^{c'}\tilde{\beta}\|+\|\sum\limits_{i=1}^{n}\mathbf{W}_i^cY_i^c\| \\
	& = A_1+A_2+A_3+A_4
	\end{align*}
	We now examine each of these terms. Using Assumption \ref{beta},\,\ref{omega}, 
	\begin{align*}
	&A_1^2 =\|\sum\limits_{i=1}^{n} \mathbf{W}_i^cY_i^c\tilde{\beta}^{'}\Omega\tilde{\beta} \|^2 \leq  c \left(\sum\limits_{i=1}^{n} \|\mathbf{W}_i^c\|\|Y_i^c\|\right)^2 \leq c \sum\limits_{i=1}^{n} \|\mathbf{W}_i^c\|^2\sum\limits_{i=1}^{n} \|\mathbf{Y}_i^c\|^2 \\
	&\sum\limits_{i=1}^{n} E\|\mathbf{W}_i^c\|^2 = \sum\limits_{i=1}^{n} E\|\mathbf{W}_i -\overline{\mathbf{W}}\|^2\leq \sum\limits_{i=1}^{n}  
	E \|\mathbf{W}_i\|^2+2n\|\overline{\mathbf{X}}\|^2+2n\|\overline{\mathbf{U}}\|^2.
	\end{align*}
	Consider,
	\begin{align*}
	&\sum\limits_{i=1}^{n}  E \|\mathbf{W}_i\|^2 \leq  \sum\limits_{i=1}^{n} \| \mathbf{X}_i\|^2+ E\|\mathbf{U}_i\|^2 \leq cn+ \sum\limits_{i=1}^{n}\sum\limits_{l=1}^{p_n} E(u_{il}^2)=cn+\sum\limits_{i=1}^{n}\sum\limits_{l=1}^{p_n} \lambda_l =O(n)\\
	& nE\|\overline{\mathbf{U}}\|^2 = nE\| \sum\limits_{i=1}^{n} \dfrac{\mathbf{U}_i}{n} \|^2 = n^{-1}E\left(\sum\limits_{l=1}^{p_n} \left(\sum\limits_{i=1}^{n} u_{il} \right)^2 \right)=n^{-1} \sum\limits_{l=1}^{p_n}\sum\limits_{i=1}^{n} Eu_{il}^2+0=O(1).
	\end{align*}
	Using the fact that $\|\overline{\mathbf{X}}\|=O(1)$ and $2\sum\limits_{i=1}^{n} \|\mathbf{Y}_i^c\|^2 =O_p(n)  $  we get $A_1=O_p(\sqrt{n})$. Similarly we can prove that the other terms are of the same order thus proving the result.
\end{proof}

Let $J(\beta)=\left(\dfrac{\partial U(\beta)}{\partial \beta}\right)$.

\begin{proposition}\label{c_prop2}
	$\underset{\| \beta-\tilde{\beta}\|= \zeta \sqrt{p_n/n}}{\sup}
	\underset{\|b\|=1}{\sup} | b^{'} (J(\beta)-J(\tilde{\beta} )b | \leq O_p(\sqrt{np_n})$
	
\end{proposition}

\begin{proof}
	We can easily show that
	$$J(\beta)=-\sum\limits_{i=1}^{n}\Omega(\mathbf{W^c}_i'\beta)Y_i^c-\Omega\beta \mathbf{W}_i^{c'}Y_i^c+\Omega Y_i^{c2}-\mathbf{W}_i^c\mathbf{W}_i^{c'}. $$
	Then,
	\begin{align*}
	b^{'}(J(\beta)-J(\tilde{\beta}))b &=b^{'}\left(\sum\limits_{i=1}^{n} -2\Omega(\mathbf{W}_i^{c'}\beta)Y_i^c+\Omega Y_i^{c2}-\mathbf{W}_i^c\mathbf{W}_i^{c'}\right)b\\
	& \quad -b^{'}\left(\sum\limits_{i=1}^{n} -2\Omega(\mathbf{W}_i^{c'}\tilde{\beta})Y_i^c+\Omega Y_i^{c2}-\mathbf{W}_i^c\mathbf{W}_i^{c'}\right)b\\
	&=b^{'}\left( \sum\limits_{i=1}^{n} 2\Omega(\mathbf{W}_i^{c'}\tilde{\beta})Y_i^c-
	2\Omega(\mathbf{W}_i^{c'}\beta)Y_i^c\right) b\\
	&= b^{'}\left( \sum\limits_{i=1}^{n} 2Y_i^c\Omega\mathbf{W}_i^{c'}(\tilde{\beta}-\beta) \right)b\\
	|b^{'}(J(\beta)-J(\tilde{\beta}))b| & \leq \sum\limits_{i=1}^{n} | b^{'}2Y_i^c\Omega\mathbf{W}_i^{c'}(\tilde{\beta}-\beta)b| \\
	& \leq 2\|\Omega\| \|\tilde{\beta}-\beta\|\|b\|^2  \sum\limits_{i=1}^{n} \|Y_i^c\|\|
	\mathbf{W}_i^{c'}\| .
	\end{align*}
	Taking supremum
	\begin{align*}
	\underset{\| \beta-\tilde{\beta}\|= \zeta \sqrt{p_n/n}}{\sup}
	\underset{\|b\|=1}{\sup}| b^{'}(J(\beta)-J(\tilde{\beta}))b| \leq O_p(\sqrt{np_n}) .
	\end{align*}

\end{proof}

\begin{proposition}\label{c_prop3}
	$$\underset{\| \beta-\tilde{\beta}\|= \zeta \sqrt{p_n/n}}{sup} (\beta-\tilde{\beta})^{'}J(\tilde{\beta})(\beta-\tilde{\beta})<-c\zeta^2 p_n+ \zeta^2o_p(p_n) $$
\end{proposition}
\begin{proof}
	Let $b= \beta-\tilde{\beta}$. Consider,
	
	\begin{align*}
	J(\tilde{\beta})=-2\sum\limits_{i=1}^{n}\Omega(\mathbf{W^c}_i'\tilde{\beta})Y_i^c+\Omega Y_i^{c2}-\mathbf{W}_i^c\mathbf{W}_i^{c'}
	\end{align*}
	We evaluate each term.
	\begin{align*}
	-2\sum\limits_{i=1}^{n}\Omega(\mathbf{W^c}_i'\tilde{\beta})Y_i^c  & =-2\Omega\tilde{\beta}^{'}\sum\limits_{i=1}^{n}Y_i^c(\mathbf{X}_i^c+\mathbf{U}_i^c)\\
	&=-2\Omega\tilde{\beta}^{'}n\sum\limits_{i=1}^{n}\dfrac{Y_i^c\mathbf{X}_i^c}{n}-2\Omega\tilde{\beta}^{'}n\sum\limits_{i=1}^{n}\dfrac{Y_i^c\mathbf{U}_i^c}{n}\\
	&=-2\Omega\tilde{\beta}^{'}n\sum\limits_{i=1}^{n}\dfrac{(\mathbf{X}_i^{c'}\tilde{\beta})^2}{n}-2\Omega\tilde{\beta}^{'}nE(Y_1^c\mathbf{U}_1^c) +\text{smaller order terms}\\
	&=-2\Omega\tilde{\beta}^{'}n\sum\limits_{i=1}^{n}\dfrac{\tilde{\beta}^{'}\mathbf{X}_i^c\mathbf{X}_i^{c'}\tilde{\beta}}{n}-2\Omega co_p(n) \\
	-b^{'}2\sum\limits_{i=1}^{n}\Omega(\mathbf{W^c}_i'\tilde{\beta})Y_i^cb &\leq
	-c\zeta^2p_n- \zeta^2o_p(p_n)
	\end{align*}
	Next consider,
	\begin{align*}
	b^{'}\sum\limits_{i=1}^{n} \Omega Y_i^{c2}b &=b^{'}\Omega b\sum\limits_{i=1}^{n} Y_i^{c2}=b^{'}\Omega b n \text{var}(Y_1)+\text{smaller order terms} \leq \zeta^2p_n \lambda_{max}(\Omega_1).
	\end{align*}
	Finally consider,
	\begin{align*}
	-n\sum\limits_{i=1}^{n}\dfrac{\mathbf{W}_i^c\mathbf{W}_i^{c'}}{n} &=
	-n\sum\limits_{i=1}^{n}\dfrac{(\mathbf{X}_i^c+\mathbf{U}_i^c)(\mathbf{X}_i^{c'}+\mathbf{U}_i^{c'})}{n}\\
	&= -n\sum\limits_{i=1}^{n} \dfrac{\mathbf{X}_i^c\mathbf{X}_i^{c'}}{n}-n
	\dfrac{\mathbf{X}_i^c\mathbf{U}_i^{c'}}{n}-n\dfrac{\mathbf{U}_i^c\mathbf{X}_i^{c'}}{n}-n\dfrac{\mathbf{U}_i^c\mathbf{U}_i^{c'}}{n}\\
	&=-n\sum\limits_{i=1}^{n} \dfrac{\mathbf{X}_i^c\mathbf{X}_i^{c'}}{n}- \Omega_1+\text{smaller order terms} \\
	-n b^{'}\sum\limits_{i=1}^{n}\dfrac{\mathbf{W}_i^c\mathbf{W}_i^{c'}}{n} b &=-nb^{'}\sum\limits_{i=1}^{n} \dfrac{\mathbf{X}_i^c\mathbf{X}_i^{c'}}{n} b-b^{'}\Omega_1 b+\text{smaller order terms} \\
	&\leq -\zeta^2 p_n\lambda_{min}\left(\sum\limits_{i=1}^{n} \dfrac{\mathbf{X}_i^c\mathbf{X}_i^{c'}}{n} \right)-\zeta^2 p_n\lambda_{min}(\Omega_1).
	\end{align*}
	
	From Assumption \ref{vary},\, we obtain he desired result.
\end{proof}


\vskip .65cm
\noindent
Yale University
\vskip 2pt
\noindent
E-mail: (sneha.jadhav@yale.edu)
\vskip 2pt

\noindent
Yale University
\vskip 2pt
\noindent
E-mail: (shuangge.ma@yale.edu)
\end{document}